\documentclass[11pt]{article}
\usepackage{amssymb}
\usepackage{amsmath}
\usepackage{amsthm}
\usepackage{cleveref}
\usepackage{amsfonts}
\usepackage{graphicx}
\usepackage{placeins}
\usepackage{comment}
\usepackage{esint}
\usepackage{enumerate}
\usepackage{mathtools}
\usepackage[margin=1in]{geometry}
\usepackage{xcolor}

\newcommand{\tr}[1]{\mathrm{tr}\,{#1}}

\newcommand{\e}{\varepsilon}
\newcommand{\R}{\mathbb R}

\newcommand{\beq}{\begin{equation}}
\newcommand{\eeq}{\end{equation}}

\newcommand{\dive}{\mathrm{div}\,}
\newcommand{\curl}{\mathrm{curl}\,}
\newcommand{\dz}{\partial_z}
\newcommand{\dx}{\partial_x}
\newcommand{\dy}{\partial_y}
\newcommand{\dxx}{\dx^2}

\newcommand{\dxz}{\dx\dz}
\newcommand{\dzx}{\dz\dx}

\newcommand{\ro}{\mathrm{R}\,}
\newcommand{\mn}{m_n}

\renewcommand\mp{m^+}

\renewcommand\mn{m_n}
\newcommand{\mm}{m^-}
\newcommand{\mmo}{\mm_1}
\newcommand{\mpo}{\mp_1}
\newcommand{\mres}{\mathbin{\vrule height 1.6ex depth 0pt width
0.13ex\vrule height 0.13ex depth 0pt width 1.3ex}}
\newcommand{\parallelsum}{\mathbin{\!/\mkern-5mu/\!}}

\newtheorem{theorem}{Theorem}[section]

\newtheorem{proposition}[theorem]{Proposition}
\newtheorem{lemma}[theorem]{Lemma}

\newtheorem{remark}[theorem]{Remark}

\title{\bf Compactness and sharp lower bound  for a 2D smectics model}

\author{Michael Novack\thanks{%
Department of Mathematics, The University of Texas at Austin, Austin, TX, USA }
\and Xiaodong Yan\thanks{%
Department of Mathematics, The University of Connecticut, Storrs, CT, USA } }

\begin{document}

\maketitle

\begin{abstract}
We consider a 2D smectics model%
\begin{equation*}
E_{\epsilon }\left( u\right) =\frac{1}{2}\int_\Omega \frac{1}{\varepsilon }\left( u_{z}-\frac{1%
}{2}u_{x}^{2}\right) ^{2}+\varepsilon \left( u_{xx}\right) ^{2}dx\,dz.
\end{equation*}%
For $\varepsilon _{n}\rightarrow 0$ and a sequence $\left\{ u_{n}\right\} $
with bounded energies $E_{\varepsilon _{n}}\left( u_{n}\right) ,$
we prove compactness of {\color{black}$\{\dz u_{n}\}$ in $L^{2}$ and $\{\dx u_n\}$ in $L^q$ for any $1\leq q<p$ under the additional assumption $\| \dx u_{n}\| _{L^{p }}\leq C$ for some $p>6$}. We also prove a sharp lower bound on $E_{\varepsilon }$ when $\varepsilon\rightarrow 0.$ The sharp bound corresponds to the energy of a 1D ansatz in the transition region.
\end{abstract}


\section{Introduction}\label{sec:intro}

Liquid crystal phases occur when a material exhibits characteristics of a crystalline solid while also retaining the ability to flow like a liquid. Smectic-A liquid crystals (smectics) consist of a stack of uniformly spaced layers of liquid which forms a one dimensional density
wave. The molecules in each layer tend to align in the direction parallel to the layer normal. Smectics are typically described \cite{CL95, deGP} by a complex order parameter $%
\Psi ,$ where the magnitude of $\Psi$ describes the smectic order and
the level sets of the phase $\Phi =\textup{Arg }\Psi $ determine the smectic layers. This is achieved through the introduction of the molecular mass density $\rho$, given at a point $\bold{x}=(x,y,z)$ by
\begin{equation*}
\rho =\rho _{0}+\rho _{1}\cos \left[ \frac{2\pi }{a}\Phi \left( \bold{x}\right) %
\right] ,
\end{equation*}%
where $\rho _{0}$ is a locally uniform mass density, $\rho _{1}$ is the density of the layers, and $a$ is the uniform spacing between layers. Smectic layers are defined as
peaks of the density wave where $\Phi \left( \bold{x}\right) \in a\mathbf{Z.}$

The free energy of a smectic liquid crystal \cite{SanKam05} over a sample volume $\Omega,$
expressed in terms of the phase $\Phi ,$ is 
\begin{equation}
F=\frac{B}{2}\int_{\Omega}\left[ \left( 1-\left\vert \nabla \Phi \right\vert
\right) ^{2}+\lambda ^{2}\mathbf{H}^{2}\right] dx\,dy \,dz,
\label{smectics-freeenergy}
\end{equation}%
where $\mathbf{H=\nabla \cdot N}$ is the mean curvature and $\mathbf{N}=%
\frac{\nabla \Phi }{\left\vert \nabla \Phi \right\vert }$ is the unit normal
vector for the layers. Here $B$ and $K_{1}=B\lambda ^{2}$ are the
bulk and bend moduli, respectively. The constant $\lambda $ (bend penetration depth) is the intrinsic length scale that sets the scale of deformations. The first term in $\left( \ref{smectics-freeenergy}\right) $ accounts
for the compression strain between layers and the second term represents the
bending energy. When boundaries are present, there is an additional term
coming from the Gaussian curvature 
\begin{equation*}
F_{K}=\widetilde{K}\int_{\Omega}\nabla \cdot \left[ \left( \nabla \cdot \mathbf{N}%
\right) \mathbf{N-}\left( \mathbf{N\cdot \nabla }\right) \mathbf{N}\right]
dx\,dy\,dz.
\end{equation*}%
Since this term is a total derivative which reduces to a boundary integral
and does not play a role in minimization of the energy under fixed boundary conditions, it is often omitted from the free energy.

Inspection of the total energy \eqref{smectics-freeenergy} reveals that the compression term prefers equally spaced layers while the bending term prefers layers with zero mean
curvature, which are minimal surfaces. The typical ground state minimizing \eqref{smectics-freeenergy} is $\Phi(\bold{x}) = \bold{x} \cdot n$ for fixed $n\in \mathbb{S}^2$, which corresponds to uniformly spaced layers perpendicular to the $n$ direction. However, boundary conditions can impose curvature on the layers and the resulting curvature is generally incompatible with equally spaced layers. Due to the intrinsic interplay between the layer
spacing, the Gaussian curvature, and the mean curvature, the problem of finding
minimal configurations for the energy $\left( \ref{smectics-freeenergy}\right) $ is
challenging. Over the years, physicists have devoted significant effort to looking
for exact or approximate solutions of deformations in smectic liquid
crystals \cite{AleKamSan12,AleCheMatKam10,AleKamMos12, BluKam02, BreMar99, DiK02, DiK03, IL99, KamLub99, KamSan06,
MatKamSan12, San06, SanKam03, SanKam05, SanKam07}. {\color{black}The goal of this article is to analyze smectics using tools from the mathematical theory for similar singularly perturbed variational problems, thus providing a link between these ideas and the extensive physics literature.}\par
In order to provide the necessary background for our analysis, let us briefly review the relevant work from this literature. To study deformations of the smectic layers, we fix coordinates by choosing $n=\hat{z}$ and introduce the Eulerian displacement field
$$ u(\bold{x}) = z - \Phi(\bold{x}).$$
Expanding the compression strain in powers of $%
\nabla u=\hat{z}-\nabla{\Phi} $ and defining $\nabla _{\bot }$ $=\dx\hat{x}+\dy\hat{y},$ we can
write the compression strain as
\begin{equation*}
1-\left\vert \nabla \Phi \right\vert \approx \partial _{z}u-\frac{1}{2}%
\left| \nabla _{\bot }u\right| ^{2}+\mathcal{O}\left( |\nabla u|^{3}\right) .
\end{equation*}%
In the limit of small elastic strains $\left\vert \nabla u\right\vert \ll 1,$
it is typical to retain only the terms quadratic in derivatives of $u$ in $%
\left( \ref{smectics-freeenergy}\right) $, yielding a linear theory to
describe elastic deformations in smectic liquid crystals \cite{deGP, Kle83}. While the linear
theory is rather successful in describing deformations for screw
dislocations and small angle twist grain boundaries, it misses much of the
essential physics for edge dislocations or large angle twist grain
boundaries. Due to the truncation at the level of quadratic terms in $\nabla u$, the linear theory for edge dislocations in smectics is only valid in
the limit $\frac{\left\vert \nabla _{\perp }u\right\vert ^{2}}{\left\vert
\partial _{z}u\right\vert }\ll 1$. This ratio is of order $b/\lambda $, where 
$b\in a\mathbf{Z}$ is the Burgers vector, and is not small for an edge dislocation \cite{SanKam03}. This was first observed by Brener and Marchenko \cite{BreMar99}, who demonstrated that for the case $b\sim \lambda ,$ nonlinear effects must be taken
into account to describe the asymptotic behavior even far from the defect
core where elastic strain is small. They found an exact solution to the Euler-Lagrange equation for the
following nonlinear approximation of $\left( \ref{smectics-freeenergy}%
\right) $ in two dimensions in the regime $\partial _{z}u\sim \left(
\partial _{x}u\right) ^{2}\ll 1$

\begin{equation}
F=\frac{B}{2} \int_{\Omega}\left[\left( \partial _{z}u-\frac{1}{2}\left( \partial
_{x}u\right)^2 \right) ^{2}+\lambda ^{2}\left( \partial _{x}^{2}u\right) ^{2}%
\right] dx\,dz.  \label{2dnonlinear}
\end{equation}%
Their solution differs significantly from the linear profile even far from
the defects where the elastic strain and layer curvature is small. In the limit of large
bending rigidity, Brener and Marchenko's solution recovers the profile from linear theory. Their construction was confirmed experimentally by Ishikawa and Lavrentovich 
\cite{IL99} in a cholesteric finger texture.\par
Brener and Marchenko's solution is a specific example of a special class of exact solutions for nonlinear approximations
of (\ref{smectics-freeenergy}) developed later by Santangelo and Kamien \cite%
{SanKam03}. They studied the 3D nonlinear approximation of $%
\left( \ref{smectics-freeenergy}\right) $%
\begin{equation}
F=\frac{B}{2} \int_{\Omega}\left[\left( \partial _{z}u-\frac{1}{2}\left| \nabla
_{\bot }u\right| ^{2}\right) ^{2}+\lambda ^{2}\left( \Delta _{\bot }u\right)
^{2}\right] dx, {\color{black}\label{3dnonlinear}}
\end{equation}%
where $\nabla _{\bot }u=\left( \partial _{x}u,\partial _{y}u\right) $ and $\Delta _{\bot }u=\partial
_{x}^{2}u_{{}}+\partial _{y}^{2}u_{{}}$ is the linear approximation of the
mean curvature $\mathbf{H=\nabla \cdot N.}$ Following a method developed by
Bogomol'nyi \cite{Bo76}, Prasad and Sommerfield \cite{PraSom76} (BPS
decomposition)\ in the study of field configurations of magnetic monopoles
and solitons in field theory, Santangelo and Kamien decomposed the total
energy $\left(\ref{3dnonlinear}\right) $ into the sum of a perfect square
and a total derivative plus an additional term $\int u\overline{K}$, where 
\begin{equation*}
\overline{K}=\frac{1}{2}\nabla _{\bot }\cdot \left( \nabla _{\bot }u\Delta
_{\bot }u-\frac{1}{2}\nabla _{\bot }\left\vert \nabla _{\bot }u\right\vert
^{2}\right)
\end{equation*}%
represents the approximation of Gaussian curvature in terms of the Eulerian
displacement $u$. For deformations with $\overline{K}=0$, the free energy
reduces to the sum of a perfect square plus a series of surface terms. The minimum is therefore achieved by BPS solutions where the perfect square term vanishes. The
BPS solutions satisfy a nonlinear differential equation of reduced order
which can be transformed into a linear equation through a Hopf-Cole
transformation. The energy of these configurations simplifies to a
topological term which can be evaluated on the layers near defect core. For
a deformation depending only on $z$ and $x,$ so that $\overline{K}=0$, the BPS
equation becomes 
\begin{equation}
\partial _{z}u-\frac{1}{2}\left( \partial _{x}u\right) ^{2}-\lambda \partial
_{x}^{2}u=0,  \label{bpsequation}
\end{equation}%
which recovers Brener and Marchenko's solution through the boundary constraint $u_{\pm }\left(
x,z=0\right) =\pm \frac{b}{2}\Theta \left( x\right) $, where $\Theta \left(
x\right) $ is the step function. For small $\overline{K},$ the BPS solutions
exhibit lower energy than profiles from the linear theory \cite{SanKam03}. The same
approach was generalized by Santangelo and Kamien \cite{SanKam05} to the full smectic energy \eqref{smectics-freeenergy} where they identified a special class of minima when
Gaussian curvature vanishes. In particular, their analysis showed that the layer
deformation in the full theory is very close to that from the partially
nonlinear theory studied in \cite{BreMar99} and \cite{SanKam03}.

While many physics papers focus on finding exact solutions or approximate solutions for nonlinear smectics, {\color{black}works on the mathematical analysis of similar models often focus on the asymptotic behavior of the energy as a small parameter such as $\lambda$ approaches zero. In particular, proving compactness and convergence to a limiting energy (in the sense of $\Gamma$-convergence) are natural questions in light of the fact that $\Gamma$-convergence and equicoercivity imply the convergence of minimizers to minimizers. Since many of our techniques draw from these ideas, it is perhaps instructive to recall a selection of results for a well-studied example, the Aviles-Giga functional in two dimensions. Aviles and Giga \cite{AviGig87} proposed the energy 
\begin{equation}\label{AGoriginal}
\mathcal{F}_{\varepsilon}=\int_{\Omega}\frac{1}{\varepsilon}(|\nabla u|^2-1)^2+\varepsilon(\Delta u)^2
\end{equation}
as a model of liquid crystals in the smectic state, where $\Omega \subset \mathbb{R}^2$ is a bounded domain. Observe this is similar to  \eqref{smectics-freeenergy} with $\Phi=u$ and the mean curvature $\nabla \cdot \frac{\nabla u}{|\nabla u|}$ replaced by its linear approximation $\Delta u$. Up to a boundary term, \eqref{AGoriginal} is equivalent to 
\begin{equation}\label{Aviles-Giga}
\mathcal{E}_{\varepsilon}=\int_{\Omega}\frac{1}{\varepsilon}(|\nabla u|^2-1)^2+\varepsilon|\nabla \nabla u|^2.
\end{equation}
When $\varepsilon$ goes to zero, it is expected that minimizers of $\mathcal{E}_{\varepsilon}$ (subject to suitable boundary constraints) converge (in a suitable Sobolev space) to a limiting state $u_0$ which represents the smectic state and satisfies eikonal equation $|\nabla u|=1$ a.e. in $\Omega$. Due to boundary constraints, the solutions to eikonal equation are not smooth and the limiting energy concentrates on the discontinuities of $\nabla u$ (the folds).  The singular perturbation term $\varepsilon|\nabla \nabla u|^2$ provides a selection mechanism, choosing a special fold-energy minimizing solution of the eikonal equation. Aviles and Giga conjectured the fold-energy corresponds to a 1D ansatz at the $\e$ level and the limiting energy takes the form
\begin{equation*}
\mathcal{E}_0=1/3\int_{J_{\nabla u}}|[\nabla u]|^3,
\end{equation*}
where the limiting function $u$ satisfies eikonal equation $|\nabla u|=1$ a.e., $J_{\nabla u}$ is the defect set, and $[\nabla u]$ is the jump in  $\nabla u$ across $J_{\nabla u}$. The first significant progress toward Aviles-Giga's conjecture came in the work of Jin and Kohn \cite{JinKoh00}, where they developed a new scheme for proving lower bounds for the Aviles-Giga energy \eqref{Aviles-Giga} in two dimensions. Jin and Kohn  observed that the divergence of vector field
\begin{equation}\label{jkentropy}
\Sigma u=(u_1(1-u_2^2-\frac{1}{3}u_1^2),-u_2(1-u_1^2-\frac{1}{3}u_2^2)),      u_i=\partial_{i}u,    u_{ii}=\partial_{ii}u,
\end{equation}
namely $(1-|\nabla u|^2)(u_{11}-u_{22})$, can be used to bound $\mathcal{E}_{\varepsilon}$  from below. Under the specific choice of boundary conditions $u=0$, $\frac{\partial u}{\partial n}=-1$ on $\partial \Omega$,  they showed the lower bounds from $\dive \Sigma u$ are asymptotically sharp for certain domains, supporting Aviles-Giga's conjecture that the optimal transition layers are one dimensional. The picture in two dimensions was completed by Ignat and Monteil \cite{IgnMon20} where they proved any minimizer of \eqref{Aviles-Giga} on an infinite strip is one-dimensional.  Taking the supremum of the divergences of  all rotated variants of $\Sigma u$, Aviles and Giga \cite{AviGig99} obtained a limiting functional $J: W^{1,3}(\Omega)\rightarrow [0, \infty)$ which satisfies
$$
J(u)\leq \liminf_{n\rightarrow \infty} \mathcal{E}_{\varepsilon_n}(u_n)
$$ 
for any sequence $u_n$ converging to $u$ strongly in $W^{1,3}(\Omega)$.  Moreover, $J$ is lower semicontinuous with respect to strong convergence in $W^{1,3}(\Omega)$ and coincides with $\mathcal{E}_0$ for any $u$ satisfying eikonal equation with $\nabla u\in BV(\Omega)$. The matching upper bound when $\nabla u \in BV(\Omega)$ was shown in \cite{ConDeL07, Pol08}. 
While the  result in \cite{AviGig99} suggests the natural function space for the limiting problem is 
$$
AG_e(\Omega)=\{u\in W^{1,3}(\Omega): |\nabla u|=1\text{  a.e. in }\Omega \text{ and } J(u)<\infty\},
$$
$\mathcal{E}_0$ is well defined  only if $\nabla u$ has locally bounded variation in $\Omega$. In fact, a counterexample was constructed in  \cite{AmbDeLMan99}  showing there is a function  in $AG_e(\Omega)$ which does not have locally bounded variation in $\Omega$. In the same paper, the authors  proved compactness of the sublevel sets 
$$
{u\in AG_e(\Omega): J(u) \leq M}
$$
for any constant $M>0$ and equicoercivity of $\mathcal{E}_{\varepsilon}$, i.e. any sequence $\{u_n\}$ with $\mathcal{E}_{\varepsilon_n}(u_n)$ bounded and $\varepsilon_n \downarrow 0$ has a subsequence converging to a limiting function $u\in AG_e(\Omega)$. 
By  a different argument,  the authors in \cite{DKMO01} proved that the gradients of a  sequence $\{u_{\varepsilon}\}$ with bounded energy $\mathcal{E}_{\varepsilon}$ as $\varepsilon$ goes to zero are compact in $L^2$.
}

{\color{black}Motivated by the physics literature on smectic liquid crystals and analysis tools developed in the study of Aviles-Giga problem}, we consider the
2D nonlinear approximations of (\ref{smectics-freeenergy}) studied by Brener
and Marchenko \cite{BreMar99}
\begin{equation*}
F\left( u\right) =\frac{1}{2}\int_{\Omega}\left[ B\left( \partial _{z}u-\frac{1}{2%
}\left( \partial _{x}u\right) ^{2}\right) ^{2}+K_{1}\left( \partial
_{x}^{2}u\right) ^{2}\right] dx.
\end{equation*}%
Setting $\e = \sqrt{K_1/B}$ and multiplying through by $(B\e)^{-1}$, we arrive at
\begin{equation}
E_{\varepsilon }\left( u\right) =\frac{1}{2}\int_{\Omega }\left[ \frac{1}{%
\varepsilon }\left( \partial _{z}u-\frac{1}{2}\left( \partial _{x}u\right)
^{2}\right) ^{2}+\varepsilon \left( \partial _{x}^{2}u\right) ^{2}\right]
dx\,dz,  \label{2dsingular}
\end{equation}%
where $\Omega \subset \mathbb{R}^{2}$ is a bounded region. We are interested in the asymptotic behavior of $\inf_u E_\e$ as $\varepsilon \rightarrow 0$, which corresponds
to the physical case where the intrinsic length scale (the bend penetration depth $\e = \lambda$) is vanishingly small compared to a length scale related to the problem geometry (size of $\Omega$).\par
Our main results are:
\vspace{.1cm}\begin{itemize}
\item a compactness theorem for a sequence with bounded energies (\Cref{2dcpt}),
\item a lower bound on $E_{\varepsilon }$ when $\varepsilon \rightarrow
0$ (\Cref{2dlbd}), and 
\item {\color{black}a sharp upper bound when $\nabla u \in BV(\Omega)\cap L^\infty(\Omega)$ (\Cref{2dupbd}).}
\end{itemize}\vspace{.1cm}
For $\varepsilon _{n}\rightarrow 0$ and a sequence $\left\{
u_{n}\right\} $ with bounded energies $E_{\varepsilon _{n}}\left(
u_{n}\right)$, we prove compactness of {\color{black}$\{\dx u_n\}$ in $L^q$ for any $1\leq q<p$ and compactness of $\{\dz u_n\}$  in $L^{2}$} under the additional assumption $\| \dx u_{n}\| _{L^{\color{black}p }}\leq
C $ {\color{black}  for some $p >6$}. This assumption is physically justifiable since the model \eqref{2dsingular} is only valid in the limit of small strains \cite{BreMar99,SanKam03}. {\color{black}From a mathematical perspective, some assumption is necessary for a compactness result due to the fact that the set $\{m_2 = m_1^2/2 \}\subset \mathbb{R}^2$ is unbounded. }
Our compactness proof uses an entropy argument following the work of Tartar 
\cite{Tar79, Tar83, Tar05} and Murat \cite{Mur78, Mur81ASNP, Mur81JMPA} on compensated compactness.

For the lower bound, by applying the BPS decomposition to $\left( \ref{2dsingular}\right) ,$ we can
write $E_{\varepsilon }(u)$ as 
\begin{equation}\label{decomp}
E_\e(u) = \int_\Omega\dive \Sigma \left( \nabla u\right)\,dx \,dz + \frac{1}{2}\int_{\Omega }\frac{1}{\varepsilon }\left( \partial _{z}u-%
\frac{1}{2}\left( \partial _{x}u\right) ^{2}-\varepsilon \partial
_{x}^{2}u\right) ^{2}\,dx\,dz,
\end{equation}
cf. \eqref{bps}. It then follows that $E_{\varepsilon }$ is always bounded from below by the integral of the total derivative and is saturated when the perfect square term vanishes. However, at this point, we do not search for solutions of \eqref{bpsequation} so that the second term in \eqref{decomp} vanishes. Instead, for a sequence $\varepsilon _{n}\rightarrow 0$ and $\left\{ u_{n}\right\} $ in $ H^{1}$ converging to a limiting function $u$ in a suitable space, we use \eqref{decomp} to bound $%
\lim \inf E_{\varepsilon _{n}}\left( u_{n}\right) $ from below by the total
jump in $\Sigma \left( \nabla u\right) $ of the limit function $u$ across
the jump set, explicitly written as 
\begin{equation}\notag
\int_{J_{\nabla u}} \frac{|\dx u^+ - \dx u^-|^3}{12\sqrt{1+\frac{1}{4}(\dx u^+ +\dx u^-)^2}}d\mathcal{H}^1,
\end{equation}
cf. \eqref{lscineq}. The argument is {\color{black}strongly} reminiscent of the Jin-Kohn argument mentioned above \cite{JinKoh00}. Next, for the matching upper bound, by a general theorem of Poliakovsky \cite{Pol08}, it suffices to show the localized problem on a square is
asymptotically minimized by a 1D ansatz when $\varepsilon \rightarrow 0$.
Our 1D ansatz satisfies the BPS equation so that the perfect square term
vanishes and matches the lower bound asymptotically. {\color{black} Regarding the 3D model \eqref{3dnonlinear}, however, the BPS decomposition yields an additional term $\int u\overline{K}$ in the sum, and one can only get a lower bound when the approximation of the Gaussian curvature $\overline{K}$ vanishes. In the physics literature, this restriction has been noted in \cite{SanKam03, SanKam05}. To obtain the lower bound in the general case, a different argument is needed instead of the BPS decomposition. See \cite{NY2} for further details.} \par
Finally, we remark on the implications of the analysis here. Our work indicates that for the 2D smectic model \eqref{2dnonlinear}, the local defect energy of asymptotically minimal configurations corresponds to the energy of a 1D ansatz in which $\nabla u$ varies in the direction transverse to the defect. {\color{black}A full $\Gamma$-convergence proof would entail the construction of a recovery sequence $\{ \nabla u_\e\}$ when $\nabla u \notin BV(\Omega)\cap L^\infty(\Omega)$, which presents non-trivial technical issues, cf. \Cref{lscremark}.} The optimality/non-optimality of 1D transition regions and the possible emergence of microstructure is a recurring theme in problems coming from materials science; see for example \cite{Koh07}. It is interesting that such microstructure does not appear in this smectics model. Also, the sharp lower bound using the BPS decomposition is physically compelling, in that it  shows that minimization of the total energy occurs via an equipartition of energy between the bending and compression terms. Furthermore, it has been observed in the physics literature \cite{SanKam03} that the dependence of the BPS solution on the problem geometry makes it difficult to use the solution to gain insight into more complicated defect structures, e.g. curved deformations or multiple edge dislocations. Our analysis demonstrates that regardless of geometry, this equipartition of energy is optimal.\par

The paper is organized as follows. After a brief review on preliminaries in \Cref{sec:prelim}, we prove compactness in \Cref{sec:cpt}.
\Cref{sec:lbd} is devoted to a lower bound with a key lemma proved  in Appendix.
In \Cref{sec:upbd} we construct a 1D ansatz in a square which matches
the lower bound from \Cref{sec:lbd} when $\varepsilon \rightarrow 0$.

\section{Preliminaries}
\label{sec:prelim}
We consider the energies
\begin{equation}\label{2denergy}
    E_\e(u) =\frac{1}{2} \int_\Omega\left\{ \displaystyle\frac{\left(\dz u - \frac{1}{2}(\dx u)^2\right)^2}{\e} + \e (\dxx u)^2\right\}\,dx\,dz.
\end{equation}
Throughout the paper $\Omega \subset \mathbb{R}^2$ will be a bounded domain. In some parts, we will require mild regularity on $\partial \Omega$ which will be specified. {\color{black}We define the admissible class by }
\begin{equation}\notag\color{black}
\mathcal{A}:= \{ u \in H^1(\Omega):\dx^2 u \in L^2(\Omega) \}.
\end{equation}
{\color{black}Regarding the choice of $\mathcal{A}$, the condition $u \in H^1(\Omega)$ is not sufficient for the integral $\int (\dz u - (\dx u )^2/2)^2$ to always be finite, so we will take the convention that $E_\e(u)=\infty$ if that term is infinite.\par
With the goal of understanding what type of boundary conditions pertain to the class $\mathcal{A}$, let $\partial \Omega$ be Lipschitz. Since $\mathcal{A}\subset H^{1}(\Omega)$, we can demand
\begin{equation}\notag
u_{\partial\Omega} = g 
,\end{equation}
where $u\in \mathcal{A}$ and $g: \partial \Omega \to \mathbb{R}$ is {\color{black}in $H^{1/2}(\partial \Omega)$}. Regarding $\nabla u$, there is not enough regularity for a Dirichlet condition. However, it is possible to require that admissible competitors for $E_\e$ satisfy a condition of the type
\begin{equation}\label{bc}
\nabla u_{\partial \Omega} \cdot \tau_\Omega = h, \quad h \in H^{-1/2}(\partial \Omega)
.\end{equation}}
Here $\tau_\Omega$ is tangent to $\partial \Omega$. This is due to the fact that $\nabla u$ belongs to the space
$$H_{\curl }(\Omega;\mathbb{R}^2)=\{m \in L^2(\Omega;\mathbb{R}^2): \curl m = \dx m_2 -  \dz m_1 \in L^2(\Omega)\}.$$
In fact, $\curl(\nabla u) = 0$ in the sense of distributions. Since $\nabla u \in H_{\curl }(\Omega;\mathbb{R}^2)$, there is, in the sense of distributions, a well-defined tangential trace $\nabla u_{\partial \Omega} \cdot \tau_\Omega\in H^{-1/2}(\partial \Omega)$, cf. \cite[Ch. 1]{Tem79}. For $\phi \in H^{1/2}(\partial \Omega)$, $\nabla u_{\partial \Omega} \cdot \tau_\Omega$ acts via the integration by parts formula
\begin{align}\label{ibpform}
< \nabla u_{\partial \Omega} \cdot \tau_\Omega, \phi > &:= \int_\Omega  \left(\Phi\, \curl (\nabla u) + \nabla^\perp \Phi \cdot \nabla u \right) \, dx \,dz \\ \notag
&= \int_\Omega   \nabla^\perp \Phi \cdot \nabla u  \, dx \,dz,
\end{align}
where $\Phi$ is an $H^1(\Omega)$ extension of $\phi$ and $\nabla^\perp \Phi = (-\dz \Phi, \dx \Phi)$. The first term in the integrand in \eqref{ibpform} vanishes since $\curl(\nabla u) = 0$. This definition is independent of the choice of $\Phi$, as can be seen by an approximation argument.\par

{\color{black} For the upper and lower bounds, we work with a different set of limiting configurations $\mathcal{A}_0\subset W^{1,\infty}(\Omega)$ (defined in \Cref{sec:lbd}).} Now when $u\in \mathcal{A}_0$, so that $\nabla u \in L^\infty(\Omega;\mathbb{R}^2)$, it turns out that the distribution $\nabla u_{\partial \Omega} \cdot \tau_\Omega$ corresponds to an $L^\infty(\partial\Omega)$-function. This is a consequence of the duality $L^\infty(\partial\Omega) = (L^1(\partial \Omega))^*$. Indeed, we fix $\phi \in L^1(\partial \Omega)$ and consider a $W^{1,1}(\Omega)$ extension $\Phi$. The right hand side of \eqref{ibpform} defines a functional which is independent of the particular $\Phi$ and thus clearly linear in $\phi$. Furthermore, it is continuous due to H{\"o}lder's inequality. \par
Lastly, we remark on the question of existence of minimizers for $E_\e$. Due to the lack of control on second derivatives of $u$ other than $\dxx u$,  it is not clear how to use the direct method to find a minimizer of $E_\e$. 
\color{black} 
However, since we are interested in characterizing states with low or near-minimal energy and our compactness theorem requires only bounded energies, this issue is not a significant obstacle.

\section{Compactness}
\label{sec:cpt}
Our main result in this section is 
\begin{theorem}\label{2dcpt}
Let $\Omega \subset \R^2$ be a bounded domain. Let $\e_n \searrow 0$, $\{u_n\}\subset H^1(\Omega)$ be such that 
\begin{equation}\notag
   \|\dx u_n \|_{L^{\color{black}p}}\leq C {\color{black} \text { for some }p >6},\quad \dx^2 u_n \in L^2(\Omega),\quad \textit{and} \quad E_{\e_n}(u_n)\leq C.
\end{equation}
Then
{\color{black}
\begin{equation}\notag
    \{\dx u_n \} \textit{ is relatively compact in } L^q(\Omega) \text{ for any } 1\leq q<p
\end{equation}
and
\begin{equation}\notag
\{\dz u_n \} \text{ is relatively compact in } L^2(\Omega)  .
\end{equation}}
\end{theorem}
\Cref{2dcpt} is a direct corollary of the following stronger proposition. To state and prove the proposition, we work with the divergence-free vector fields $\mn=(m_{n1},m_{n2})=\ro \nabla u_n$, where 
$$
\ro\nabla u_n = \begin{pmatrix}
\dx u_n \\
-\dz u_n 
\end{pmatrix}.
$$
\begin{proposition}\label{cptprop}
Let $\Omega \subset \R^2$ be a bounded domain. Let $\{ \mn \}\subset L^2(\Omega;\R^2)$ be such that 
\begin{equation}\label{divfree}
    \dive \mn = {\color{black} m_{n1z}+m_{n2x}}=0 \quad \textit{ in the sense of distributions}
,\end{equation}
\begin{equation}\label{potbd}
    \left\|m_{n2}+\tfrac{1}{2}m_{n1}^2 \right\|_{L^2} \to 0
,\end{equation}
\begin{equation}\label{enbd}
 \dx m_{n1}\in L^2(\Omega) \textit{ with }   \|\dx m_{n1} \|_{L^2}\left\| m_{n2}+\tfrac{1}{2}m_{n1}^2\right\|_{L^2}\leq C
,\end{equation}
\begin{equation}\label{lpbound}
  \textit{and}\quad  \|m_{n1} \|_{L^{\color{black}p}} \leq C {\color{black}\text{ for some } p>6.}
\end{equation}
Then{\color{black}

\begin{equation}\notag
\{m_{n1} \} \textit{ is relatively compact in }L^q(\Omega) \text{ for any } 1\leq q<p.
\end{equation}
and 
\begin{equation}\notag
    \{m_{n2} \} \text{is relatively compact in } L^2(\Omega) 
\end{equation}}

\end{proposition}
The proof of \Cref{cptprop} utilizes the compensated compactness approach of Tartar \cite{Tar79,Tar83} and Murat \cite{Mur81ASNP}.  {\color{black} Recall for a scalar conservation law 
\begin{equation*}
\partial_zs+\partial_xf(s)=0 ,
\end{equation*}
where $f$ is a $C^1$ function, an entropy solution $s$ is defined such that for any nonlinear pair $(\eta, q)$ satisfying $\eta'=q'f'$ (the so-called entropy entropy-flux pair) satisfies 
\begin{equation*}
\partial_z\eta(s)+\partial_xq(s) \text{ is a measure }.
\end{equation*}
A lemma of Murat \cite{Mur78} implies ${\partial_z\eta(s_n)+\partial_xq(s_n) }$ is compact in $H^{-1}$ for a uniformly bounded sequence of entropy solutions $\{s_n\}$. This allows application of Tartar and Murat's div-curl lemma \cite{Mur78, Tar79} to derive restrictions on the Young measure generated by $\{s_n\}$, yielding strong convergence of $s_n$. This type of argument was also used in \cite{DKMO01}.  For our case, the limiting function should satisfy
\begin{equation*}
u_z=\frac{1}{2}u_x^2 \text{ a.e. in } \Omega.
\end{equation*}
Taking into account of this, we arrive at a scalar conservation law
\begin{equation*}
\partial_zs+\partial_x f(s)=0 \text{  with  }f(s)=-\frac{1}{2}s^2.
\end{equation*}
As observed  by Tartar \cite{Tar87} in the case of 1D conservation laws,  the $H^{-1}$-compact entropy production of a single entropy can be used to obtain strong convergence. We shall follow the same idea to prove compactness in our case by choosing a suitable entropy $\eta$.
}
\begin{proof}
First, by approximation, if the proposition holds for a sequence of smooth functions $\{ m_n \}$, then it holds for a general sequence as in the statement. To see that this is the case, it suffices to approximate a general sequence $\{ m_n\}$ by a sequence of smooth functions such that \eqref{divfree}-\eqref{lpbound} still hold and $m_n$ converge in $L^2(\Omega;\mathbb{R}^2)$ if and only if their approximants do. Therefore, in the rest of the proof, we will assume that each $m_n$ is smooth, so as to allow for differentiation of various expressions.\par
Let $v_n={\color{black}m_{n1}}$ and $f(v_n)=-\frac{1}{2}v_n^2$. We can write the divergence free condition \eqref{divfree} in terms of $v_n$ as
\begin{equation}\label{rewrite}
    \dz v_n + \dx f(v_n) =- {\color{black}\dx\left[m_{n2}+ \tfrac{1}{2}m_{n1}^2\right]}.
\end{equation}
Since ${\color{black}\left\|m_{n2}+\tfrac{1}{2}m_{n1}^2 \right\|_{L^2} \to 0}$, the right hand side of \eqref{rewrite} is precompact in $H^{-1}(\Omega)$.\par
Next, setting $F(v_n) =v_n^3/3$ and multiplying \eqref{rewrite} by $-v_n=-m_{n1}$, we obtain
\begin{align}\label{i12}
    \dz f(v_n) + \dx F(v_n) &= {\color{black}m_{n1}\dx\left[m_{n2}+\tfrac{1}{2}m_{n1}^2 \right]}\\ \notag
    &= {\color{black}\dx\left[m_{n1}\left(m_{n2}+\tfrac{1}{2}m_{n1}^2 \right) \right] - \dx m_{n1}\left(m_{n2}+\tfrac{1}{2}m_{n1}^2\right)}\\ \notag
    &=: I_1^n+I_2^n. \notag
\end{align}
From the uniform bound \eqref{lpbound} on $m_{n2}$ and \eqref{potbd}, we see that
\begin{equation*}
  { \color{black} \left\|m_{n2}\left(m_{n2}+\tfrac{1}{2}m_{n1}^2 \right) \right\|_{L^{\color{black}q}}\to 0} {\color{black}\text{ for } q=\frac{2p}{2+p}}
\end{equation*}
Therefore, the sequence $\{I^n_1\}$ is precompact in $W^{-1,q}(\Omega)$. Moving on to $I_2^n$, we use \eqref{enbd} in order to estimate 
\begin{equation}\notag
    \|I_2^n \|_{L^1} \leq {\color{black}\| \dx m_{n1} \|_{L^2}\left\|m_{n2}+\tfrac{1}{2}m_{n1}^2 \right\|_{L^2}} \leq C.
\end{equation}
It follows that $\{I_1^n+I_2^n \}$ is precompact in $W^{-1,q}(\Omega)$ for $\color{black}q=2p/(2+p)<2$. On the other hand, from \eqref{lpbound} and the definition of $v_n$, the left hand side of \eqref{i12} is bounded in $W^{-1,{\color{black}\frac{p}{3}}}(\Omega)$. By interpolation, we conclude that $\dz f(v_n) +\dx F(v_n)$ is precompact in $H^{-1}(\Omega)$.\par
Now, by \eqref{lpbound}, after restricting to a subsequence, we can assume that there exists $\overline{v} {\color{black}\in L^p(\Omega)}$, $\overline{f}{\color{black}\in L^{p/2}(\Omega)}$, and $\overline{F}\in L^{\color{black}p/3}(\Omega)$ such that
\begin{equation}\notag
v_n \rightharpoonup \overline{v},f(v_n)\rightharpoonup \overline{f}, \textup{ and } F(v_n)\rightharpoonup \overline{F}
\end{equation}
{\color{black}weakly  in $L^2(\Omega)$}. Replacing $v_n$ in the above arguments by $v_n - \overline{v}$, the results of the preceding two paragraphs immediately yield
\begin{equation}\notag
\dz[v_n-\overline{v}]+\dx[f(v_n) - f(\overline{v})]\textup{ is precompact in }H^{-1}(\Omega)
\end{equation}
and
\begin{equation}\notag
    \dz[f(v_n) -f(\overline{v})]+\dx[F(v_n) - F(\overline{v})]\textup{ is precompact in }H^{-1}(\Omega).
\end{equation}
These observations allow us to apply the div-curl lemma \cite{Mur78,Tar79} to
\begin{equation}\notag
\Phi_n=\begin{pmatrix}
v_n-\overline{v} \\
f(v_n) - f(\overline{v}) 
\end{pmatrix}, \Psi_n = \begin{pmatrix}
F(v_n)-F(\overline{v}) \\
-(f(v_n)-f(\overline{v})) 
\end{pmatrix},
\end{equation}
so that {\color{black}$$\Phi_n \cdot \Psi_n {\stackrel{\ast}\rightharpoonup} (\textup{weak} \lim \Phi_n)\cdot (\textup{weak} \lim \Psi_n) \textup{ in } \mathcal{M}(\Omega) \textup{ weak} \ast.$$} Written explicitly, this reads
\begin{align}\label{less0}
    (v_n-\overline{v})&(F(v_n)-F(\overline{v})) - (f(v_n)-f(\overline{v}))^2\\ \notag
    &{\stackrel{\ast}\rightharpoonup} \left(\overline{v} - \overline{v}\right)\left(\overline{F} - F\left(\overline{v}\right)\right) - \left(\overline{f} - f\left(\overline{v}\right)\right)^2 \\ \notag
    &= - \left(\overline{f} - f\left(\overline{v}\right)\right)^2\\ \notag
    & \leq 0.
    \end{align}
However, since $F'(v_n) = f'(v_n)^2$,
\begin{align}\notag
    0 &= \left(\int_{\overline{v}}^{v_n}f'(t)\,dt \right)^2 - \left(f(v_n) - f\left( \overline{v} \right) \right)^2\\ \notag
    & \leq \left(v_n - \overline{v} \right)\int_{\overline{v}}^{v_n} f'(t)^2\,dt - \left(f(v_n) - f\left( \overline{v} \right) \right)^2\\ \notag
    &= \left(v_n - \overline{v} \right)\left(F(v_n) - F\left(\overline{v}\right) \right)-\left(f(v_n) - f\left( \overline{v} \right) \right)^2.
\end{align}
Combined with \eqref{less0}, the previous inequality implies that $(v_n-\overline{v})(F(v_n)-F(\overline{v})) - (f(v_n)-f(\overline{v}))^2\stackrel{\ast}{\rightharpoonup}0$. Hence, by \eqref{less0} and the definition of $\overline{f}$,
$$-\tfrac{1}{2}v_n^2{\rightharpoonup} \overline{f} =  f\left(\overline{v}\right) = -\tfrac{1}{2}\overline{v}^2\textup{ in }{\color{black} \mathcal{M}(\Omega) \textup{ weak} \ast.}$$
From this it follows that
$$
\int_\Omega v_n^2\,dx\,dz \to \int_\Omega \overline{v}^2\,dx\,dz.
$$
Together with the weak convergence of $v_n$ to $v$ in $L^2(\Omega)$, the previous equation gives
\begin{equation}\label{mn2}
    m_{n2} = v_n \to \overline{v} \textup{ in }L^2(\Omega).
\end{equation}
{\color{black}For the $L^q$ convergence of $\{ m_{n1}\}$, if $1\leq q <2$, by H\"older's inequality 
\begin{align*}
||m_{n1}-m_{k1}^{2}||_{L^q(\Omega)}&\leq C(\Omega)||m_{n1}-m_{k1}^{2}||_{L^2(\Omega)} \\
&\to 0 \textup{ as $n$, $k$ }\to \infty.
\end{align*}
If $q>2$, by H\"older's inequality and the uniform $L^p$ bound \eqref{lpbound}, we have
\begin{align*}
\int_\Omega |m_{n1}-m_{k1}|^q\,dx\,dz &\leq \left(\int_\Omega |m_{n1}-m_{k1}|^{p}\,dx\,dz \right)^{\frac{q-2}{p-2}} \left( \int_\Omega |m_{n1}-m_{k1}|^{2}\,dx\,dz\right)^{\frac{p-q}{p-2}} \\
&\to 0\quad\textup{as $n$, $k$ }\to \infty. 
\end{align*}
}The $L^{2}$-convergence of $m_{n2}$ is a consequence of \eqref{potbd}, \eqref{lpbound}, and \eqref{mn2}.
\end{proof}
{\color{black}
\begin{remark}\label{remark2}
In the previous proposition, we imposed an uniform $L^p$ bound on $m_{n2}$ in our assumptions. This is necessary since the well where the potential $W$ vanishes is unbounded, since we can easily pick a sequence with $u_z=u_x^2/2=C_n \rightarrow \infty$ such that the energy is zero while the sequence is divergent. It is open as to whether such a bound could be shown in certain scenarios, for example if $\{\nabla u_n \}$ have almost minimizing energy for $E_{\e_n}$ subject to a boundary condition for which the limiting problem has a minimizer $\nabla u$ satisfying an $L^p$ bound. \par
\end{remark}
\begin{remark}\label{remark3}
If we put our energy functional in a periodic setting, we can rewrite the energy (\ref{2denergy}) in the following form
$$
E_{\varepsilon}=\int_{\mathbb{T}^2}\frac{1}{\varepsilon}(|\partial_2|^{-1}(\partial_1 m_1-\partial_2\frac{1}{2}m_1^2))^2+\varepsilon(\partial_2m_1)^2 dx
$$
with $m_1=u_x$, $\partial_1 =\partial_z$, $\partial_2=\partial_x$. Adapting ideas handling  Burgers operator $\partial_1 w-\partial_2\frac{1}{2}w^2$  from \cite{C-AOttSte07, IORT20, OttSte10}, we can derive uniform bounds on  $m_{n2}$ in suitable Besov and $L^p$  spaces in terms of the energy. Compactness and additional properties can be obtained via Fourier analysis from those estimates.  See \cite{NY3} for further details.\par
\end{remark}
}


\section{The Lower Bound}\label{sec:lbd} We consider the question of finding a limiting functional which provides a lower bound for $E_\e$ as $\e \to 0$. In order to state the theorem, we need to recall some properties of the space $BV(\Omega;\mathbb{R}^2)$ \cite[Chapter 3]{AmbFusPal00}. \par
First, we recall the BV Structure Theorem, which in our case states that for $m \in BV(\Omega;\mathbb{R}^2)$, the Radon measure $D m$ can decomposed as 
$$
Dm = D^am + D^j m + D^c m, 
$$
where all three measures are mutually singular and can be described as follows. The first and third components, $D^a m$ and $D^c m$, are the absolutely continuous part of $D m$ (with respect to Lebesgue measure) and the Cantor part, respectively. Most important for us is the jump part, $D^jm$, which can be expressed as 
$$
 (m^+ - m^-) \otimes \nu \mathcal{H}^1\mres J_m,
$$
where $J_m$ is the countably 1-rectifiable jump set of $m$, $\nu$ is orthogonal to the approximate tangent space at each point of $J_m$, and $m^+, m^-$ are the traces of $m$ from either side of $J_m$. \par
Next, we have the BV chain rule \cite{AmbDal90, Vol67}, which says that if $F \in C^1(\mathbb{R}^2;\mathbb{R}^2)$ with bounded derivatives and $m \in BV(\Omega;\mathbb{R}^2)$, then $F \circ m$ is in $BV(\Omega; \mathbb{R}^2)$ and
$$
D(F \circ m) = \nabla F(m) \nabla m \mathcal{L}^2 + \nabla{F}(\tilde{m})D^cm + (F(m^+)-F(m^-)) \otimes \nu \mathcal{H}^1 \mres J_m.
$$
Here $\tilde{m}$ is the approximate limit of $m$ and is defined off of $J_m$ (cf. \cite[Definition 3.63]{AmbFusPal00}) and $\nabla m$ is the matrix of approximate partial derivatives of $m$ defined almost everywhere. Taking the trace on both sides, we have
\begin{align}\label{tracediv}
\dive (F \circ m) = \tr(\nabla F(m) &\nabla m) \mathcal{L}^2 + \tr (\nabla{F}(\tilde{m})D^cm) \\ \notag &+ (F(m^+)-F(m^-)) \cdot \nu \mathcal{H}^1 \mres J_m
.\end{align}\par
Given those preliminary results, let us focus on the problem at hand. In the compactness result \Cref{2dcpt}, the limiting function $u$ is in $W^{1,{\color{black}2}}(\Omega)$ and satisfies $\dz u = (\dx u )^2/2 $ {\color{black} a.e.} With that in mind, we define
\begin{equation}\label{a0}
\mathcal{A}_0 := \{u \in  W^{1,\infty}(\Omega) : \nabla u \in BV(\Omega;\mathbb{R}^2) \textup{ and }\dz u = (\dx u )^2/2\textup{ a.e.}\}
.\end{equation}
{\color{black}Here we assume that the limit function $u\in  W^{1,\infty}(\Omega)$ for application of BV chain rule. }
According to the previous paragraphs, the gradient of $u \in \mathcal{A}_0$ has a jump set $J_{\nabla u}$. The traces along $J_{\nabla u}$ satisfy a jump condition due to the fact that $\nabla u$ is a gradient. Indeed, application of the divergence theorem reveals that along $J_{\nabla u}$, we have 
\begin{equation}\label{jumpcond}
\nabla u^+ \cdot \nu^\perp = \nabla u^- \cdot \nu^\perp .
\end{equation}
Next, we define the vector field 
\begin{equation}\label{sigmadef}
\Sigma(m)=(\Sigma_1(m),\Sigma_2(m)) :=\left( m_1m_2 - \frac{1}{6}m_1^3, -\frac{1}{2}m_1^2\right).
\end{equation}
A discussion of the motivation behind the definition of $\Sigma$ is below the statement of the upcoming theorem {\color{black}and in \Cref{entremark}}. {\color{black}Note that $\Sigma(\nabla u)$ is an $L^1$ function if we have $\dx u \in L^3$ and $\dz u \in L^{\frac{3}{2}}$, which motivates the choice of convergence in \Cref{2dlbd}; see also \Cref{lscremark}.} If $u\in \mathcal{A}_0$, then since $\nabla u$ is bounded and  in $BV(\Omega;\mathbb{R}^2)$, we can apply the BV chain rule and \eqref{tracediv} to $\Sigma \circ \nabla u$. A short calculation yields
\begin{align}\notag
\dive \Sigma( \nabla u) = \dx^2 u (\dz u - (\dx u)^2/2) \mathcal{L}^2 &+ ( \widetilde{\dz u} - (\widetilde{\dx u})^2/2) D^c(\nabla u)_{11}\\ \notag&+ (\Sigma(\nabla u^+) - \Sigma(\nabla u^-))\cdot \nu \mathcal{H}^1\mres J_{\nabla u}.
\end{align}
In the second term on the right hand side, $D^c(\nabla u)_{11}$ is the first entry in the first column of $D^c(\nabla u)$.\color{black}

 Since $\dz u = (\dx u)^2/2$ a.e., we find
\begin{equation}\label{limitmeasure}
\dive \Sigma( \nabla u) = (\Sigma(\nabla u^+) - \Sigma(\nabla u^-))\cdot \nu \mathcal{H}^1\mres J_{\nabla u}.
\end{equation}
As a final preliminary, we provide 
two explicit expressions 
\color{black}for $(\Sigma(\nabla u^+) - \Sigma(\nabla u^-))\cdot \nu$. The proofs can be found in the appendix.
\begin{lemma}\label{jumpexp}
Suppose $m^+$, $m^-$ satisfy $m_2 = (m_1)^2/2$, and set $p=m^+-m^-$, $n = p /|p| $, so that $m^+$, $m^-$ are admissible traces across a jump set with normal vector parallel to $ n$. Then
\begin{align}\notag
(\Sigma(\mp) - \Sigma(\mm)) \cdot n &= \frac{n_1}{2} \left(p_1p_2-m^-_1p_1^2 -\frac{1}{3}p_1^3 \right) \\ \label{jcost} 
&= \displaystyle\frac{|m_1^+ - m_1^-|^3}{12\sqrt{1 +\frac{1}{4}(m_1^+ + m_1^-)^2}}
.\end{align}
\end{lemma}
\begin{remark}\label{entremark}
The cubic growth for small jumps is the same as in the Aviles-Giga problem \cite{JinKoh00}, and is a common feature of lower bounds involving ``entropies" such as $\Sigma$ \cite{IgnMer12}. It also occurs in scalar conservation laws, where the entropy production is asymptotically cubic for small jumps \cite{IgnMer11}. {\color{black}In the former setting, the energies under consideration are of the form
$$
\int_\Omega \e |\nabla m|^2 + \frac{1}{\e} W(m) \,dx \,dz, \quad \nabla \cdot m =0,
$$
and an entropy $\Phi:\mathbb{R}^2\to \mathbb{R}^2$ is a smooth map such that $\nabla \cdot [\Phi(m)]=0$ for smooth $m$ with $\nabla \cdot m=0$ and $W(m)=0$. One can check that after setting $m=(-\dz u, \dx u)$, $\Sigma$ can be used to construct the entropy $\Phi := -\Sigma(m_2,-m_1)$. Maps of this type are also the basis of the compactness result in \cite{DKMO01}.
}\color{black}
\end{remark}\par
We now state the main theorem for this section.
\begin{theorem}\label{2dlbd}
Let $\Omega \subset \mathbb{R}^2$ be a bounded domain. Consider $\e_n \searrow 0$, $\{u_n \}\subset H^1(\Omega)$ {\color{black}with $\dx^2 u_n \in L^2(\Omega)$ such that }
\begin{equation}\label{lbdassump}
{\color{black}\dx u_n \to \dx u\textit{ in }L^3(\Omega) \quad \textit{and} \quad \dz u_n \to \dz u  \textit{ in }L^{\frac{3}{2}}(\Omega)}
\end{equation}
for some $u\in H^1(\Omega)$ with $\nabla u \in {\color{black}(L^{\infty}\cap} BV)(\Omega;\mathbb{R}^2)$. Then
\begin{equation}\label{lscineq}
\liminf_{n \to \infty} E_{\e_n}(u_n) \geq \int_{J_{\nabla u}} |(\Sigma(\nabla u^+) - \Sigma(\nabla u^-))\cdot \nu |\, d\mathcal{H}^1.
\end{equation}
\end{theorem}\par
{
Our choice of $\dive \Sigma(\nabla u)$ is motivated by the BPS decomposition. By the BPS\ decomposition, we can write $\left( \ref{2dsingular}\right) $ as 
\begin{eqnarray}
E_{\varepsilon }\left( u\right)  &=&\frac{1}{2}\int_{\Omega }\left[ \frac{1}{%
\varepsilon }\left( \partial _{z}u-\frac{1}{2}\left( \partial _{x}u\right)
^{2}\right) ^{2}+\varepsilon \left( \partial _{x}^{2}u\right) ^{2}\right]
dx\,dz  \notag \\
&=&\frac{1}{2}\int_{\Omega }\frac{1}{\varepsilon }\left( \partial _{z}u-%
\frac{1}{2}\left( \partial _{x}u\right) ^{2}-\varepsilon \partial
_{x}^{2}u\right) ^{2}\,dx \,dz+\int_{\Omega }\dive \Sigma(\nabla u)\,dz \,dz ,
\label{bps}
\end{eqnarray}%
where $\Sigma(\nabla u) =\left( \partial _{z}u\partial _{x}u-%
\frac{1}{6}\left( \partial _{x}u\right) ^{3},-\frac{1}{2}\left( \partial
_{x}u\right) ^{2}\right) .$ A direct conclusion from $\left( \ref{bps}%
\right) $ is 
\begin{equation}
E_{\varepsilon }\left( u\right) \geq \int_{\Omega }\dive \Sigma(\nabla u)    \label{lowerbd}
\end{equation}%
and $E_{\varepsilon }$ is minimized by mappings satisfying $\left( \ref%
{bpsequation}\right) .$ Bounding the energy from below by the integral of a total derivative is also the main idea of Jin and Kohn {\cite{JinKoh00}} for the Aviles-Giga problem, where the ``Jin-Kohn" entropy plays the part of $\Sigma$ above. 


Equation $\left( \ref{lowerbd}\right) $ is the starting point of our lower bound
estimate, and we use $\left( \ref{bpsequation}\right) $ in our construction
of 1D ansatz in our upper bound estimate in {\color{black}\Cref{sec:upbd}}. }\par

\begin{proof}[Proof of \Cref{2dlbd}]
We begin with the calculation
\begin{align}\label{divcalc}
\dive \Sigma(\nabla v) &= \dx \left[ \dx v \dz v - \frac{1}{6}(\dx v)^3\right] + \dz \left[ -\frac{1}{2}(\dx v)^2 \right] \\ \notag
&= \left( \dx^2v \dz v + \dx v \dzx v - \frac{1}{2}(\dx v)^2\dx^2 v - \dx v \dxz v\right) \\ \notag
&= \dx^2 v \left(\dz v - \frac{1}{2}(\dx v)^2 \right) \\ \notag
&\leq  \displaystyle\frac{\left(\dz v - \frac{1}{2}(\dx v)^2\right)^2}{2\e} + \frac{\e (\dxx v)^2}{2}
\end{align}
which holds for any smooth $v$, and hence by density any $v \in H^1(\Omega)$ with $\dx^2 v \in L^2(\Omega)$. Now if we plug in $\nabla u_n$ to \eqref{divcalc}, multiply by a test function $\varphi \in C_c^\infty(\Omega)$, and integrate by parts, we have
\begin{align}\label{divcalc2}
\int_\Omega -\Sigma(\nabla u_n) \cdot \nabla \varphi \, dx\,dz  &= \int_\Omega \dive \Sigma(\nabla u_n) \varphi \,dx \,dz \\ \notag
&\leq E_{\e_n}(u_n)\| \varphi \|_{L^\infty}.
\end{align}
As outlined in the discussion preceding the proof, we wish to take the limit as $n \to \infty$ on the left hand side of \eqref{divcalc2} to prove \eqref{lscineq}. If $\liminf E_{\e_n}(u_n)=\infty$, then \eqref{lscineq} is immediate. Therefore, we may suppose that
\begin{equation}\label{finliminf}
\liminf_{n \to \infty} E_{\e_n}(u_n) < \infty.
\end{equation}
{\color{black}By the convergence \eqref{lbdassump}, it follows that $\Sigma(\nabla u_n) \to \Sigma(\nabla u)$ in $L^1(\Omega;\mathbb{R}^2)$.} Then we can let $n \to \infty$ in \eqref{divcalc2}:
\begin{align}\label{totalvar}
\int_\Omega -\Sigma(\nabla u) \cdot \nabla \varphi \, dx\,dz &= \lim_{n\to \infty} \int_\Omega -\Sigma(\nabla u_n) \cdot \nabla \varphi \, dx\,dz \\ \notag
 &\leq \liminf_{n\to \infty} E_{\e_n}(u_n) \|\varphi \|_{L^\infty} .
\end{align}
The lower bound \eqref{lscineq} is obtained by taking the total variation of $\dive \Sigma(\nabla u)$ in \eqref{totalvar} and using the expression \eqref{limitmeasure} for $\dive \Sigma(\nabla u)$. 
\end{proof}
\begin{remark}\label{lscremark}
Upon examination of the proof of \Cref{2dlbd}, we see that if
{\color{black}$\dx u_n \to \dx u$ in $L^3$, $\dz u_n \to \dz u$ in $L^{\frac{3}{2}}$,} and $\liminf E_{\e_n}(u_n) < \infty$, then $\dive \Sigma (\nabla u)$ is a finite Radon measure and
\begin{equation}\notag
\liminf_{n \to \infty} E_{\e_n}(u_n) \geq | \dive \Sigma(u) |(\Omega).
\end{equation}
{\color{black}This argument holds even when $\nabla u \notin (L^{\infty}\cap BV)(\Omega;\mathbb{R}^2)$, since that assumption was only used when applying the $BV$ chain rule to calculate $| \dive \Sigma(\nabla u) |$. }This indicates that the space 
\begin{equation}\notag
\{u \in W^{1,{\color{black}\frac{3}{2}}}(\Omega): \dx u \in L^3,\,\dz u = (\dx u)^2/2 \textup{ and } \dive \Sigma(\nabla u) \textup{ is a Radon measure}\}
\end{equation}
is the natural limiting space for this sequence of variational problems, similar to the Aviles-Giga space defined in \cite{AmbDeLMan99}. It is possible that this space contains elements which are not in $BV(\Omega;\mathbb{R}^2)$, although we do not pursue this issue further. We refer the reader to \cite[pgs. 338-340]{AmbDeLMan99} for an example of such a map in the Aviles-Giga problem. {\color{black}This is one reason that upper bound \Cref{2dupbd} does not yield full $\Gamma$-convergence; the other is the restriction that $\nabla u \in L^\infty(\Omega)$. The current techniques involved in such constructions, developed in \cite{ConDeL07} and \cite{Pol08}, require both of these conditions on $\nabla u$, and removing them is likely non-trivial.} 
\end{remark}
\color{black}\begin{remark}\label{construct}
The identification of the cost \eqref{jcost} along defect curves could be utilized in conjunction with the geometric rigidity induced by the requirement that $\dz u = \dx u^2/2$ to compute critical configurations for the limiting energy. See for example \cite[Example 4.2]{GolNovSteVen20}, in which a critical configuration for a nematic-isotropic phase transition is calculated with the far field being given by the nematic ground state. This bears a resemblance to  some of the BPS constructions for smectics from \cite{AleKamSan12}, and such ideas could be employed in the smectic context as well.
\end{remark}\color{black}
\section{An Estimate for the Minimum Energy on a Square and the Upper Bound}\label{sec:upbd}
In light of \Cref{2dlbd}, we would like to know whether the lower bound can be matched by a construction, yielding a sharp lower bound. When $u \in \mathcal{A}_0$, we provide an affirmative answer to this question.
\begin{theorem}\label{2dupbd}
Let $u\in \mathcal{A}_0$ {\color{black}and $\partial\Omega$ be $C^2$}. Then there exists a sequence $\{ u_\e \}\subset C^2(\Omega)$ such that 
\begin{equation}\notag
u_\e \to u \textit{ in }W^{1,p}(\Omega) \textit{ for all }1\leq p < \infty
\end{equation}
and
\begin{equation}
E_\e (u_\e) \to \int_{J_{\nabla u}} |(\Sigma(\nabla u^+) - \Sigma(\nabla u^-))\cdot \nu |\, d\mathcal{H}^1
.\end{equation}
\end{theorem}\par
As a first step towards proving \Cref{2dupbd}, we will analyze a local problem for $E_\e$ posed on a square, with boundary data chosen to induce a limiting jump set parallel to two of the sides. We will show that up to an exponentially small error in $\e$, the minimum energy for the local problem is attained by a ``one-dimensional competitor" with constant gradient in the direction parallel to the jump set. Having done the analysis of the local problem, the upper bound \Cref{2dupbd} can then be shown as a consequence of a general theorem of Poliakovsky for proving upper bounds for singular perturbation problems using a one-dimensional ansatz \cite{Pol08}. The idea is that the local problem represents the cost per unit length along the jump set; this can then be made rigorous with the right tools when $\nabla u \in BV(\Omega;\mathbb{R}^2)$ \cite{ConDeL07, Pol08}. \par
For orthonormal vectors $\nu, \tau$, we consider the square $$R := \{(x,z) \in \mathbb{R}^2 : |(x,z) \cdot \nu| \leq 1/2, |(x,z) \cdot \tau | \leq 1/2  \}.$$
If we wish to force a jump set in the limit $\e \to 0$ with normal vector $\nu$, we must choose boundary data on $\{(x,z) \cdot \nu = 1/2 \}$ and $\{(x,z) \cdot \nu = -1/2 \}$ which is compatible with a jump across $\{(x,z) \cdot \nu = 0 \}$. Therefore, we choose $\mp$, $\mm$ such that $$m^+,m^-\in \{m : m_2 = m_1^2/2 \}\textup{ and }\nu \textup{ is parallel to }\mp - \mm$$and define the class
\begin{align}\notag
\mathcal{A}_R := \{u\in H^{1}(R): \nabla u = m^\pm \textup{ when }(x,z)\cdot &\nu = \pm 1/2 \textup{ and ${\color{black}\nabla}u$ is periodic}\\ \notag &\textup{with period }1 \textup{ in the }\tau \textup{ direction}    \}.
\end{align}
The restricted class of one-dimensional competitors is
\begin{align}\notag
\mathcal{A}^{1D}_R := \{u\in\mathcal{A}_R : \nabla u \cdot \tau  =m^{+} \cdot \tau=m^{-} \cdot \tau\textup{ on }R\}.
\end{align}
We set
\begin{align}\notag
r_\e = \inf_{ \mathcal{A}_R} E_\e
\end{align}
and
\begin{align}\notag
r_\e^{1D} = \inf_{ \mathcal{A}_R^{1D}} E_\e.
\end{align}
\begin{proposition}\label{boxrate}
For any $\e>0$, we have
\begin{equation}\label{relation}
|(\Sigma(m^+) - \Sigma(m^-))\cdot \nu | \leq r_\e \leq r_\e^{1D}.
\end{equation}
Furthermore, as $\e \to 0$,
\begin{equation}\label{rate}
\left| r_\e^{1D} - |(\Sigma(m^+) - \Sigma(m^-))\cdot \nu |\right| \leq c_1 e^{-c_2/ \e},
\end{equation}
where the constants $c_1$ and $c_2$ depend only on $\mp$, $\mm$, so that the one-dimensional ansatz is asymptotically minimizing and the cost is given by the jump in $\Sigma$.
\end{proposition}
\begin{proof}
The first inequality in \eqref{relation} is an immediate consequence of crucial calculation {\eqref{lscineq}} and the choice of boundary data for $\mathcal{A}_R$. The second follows since $\mathcal{A}_R^{1D} \subset \mathcal{A}_R$, so it remains to prove \eqref{rate}.\par
We construct a sequence of one-dimensional competitors $ u_\e$ such that $E_\e(u_\e)$ approaches $|(\Sigma(m^+) - \Sigma(m^-))\cdot \nu |$ at the desired rate. The techniques in such a construction are well-known in the calculus of variations, but we include a proof for the sake of completeness. For each $\e$, we will define $\nabla u_\e$ via the following ansatz:
\begin{align}\label{grads}
\nabla u_\e  &= g_\e((x,z)\cdot \nu)(\mp - \mm) + m^-\\ \notag
&= g_\e((x,z)\cdot \nu) p + \mm.
\end{align}
Here $p = \mp - \mm$ as in \Cref{jumpexp} and $g_\e:[-1/2,1/2] \to [0,1]$ is increasing and satisfies $g_\e(-1/2)=0$, $g_\e(1/2)=1$. It is easy to check that since $p$ is parallel to $\nu$,
$$
\curl  (g_\e((x,z)\cdot \nu) p + \mm)=0,
$$
so that it is possible to find $u_\e$ whose gradient is given by \eqref{grads}, and that $\nabla u_\e$ satisfies the boundary conditions required to be a member of $\mathcal{A}_R^{1D}$. Since the energy $E_\e$ does not depend explicitly on $u_\e$ but only on its gradient, for the rest of the proof, we will, with a slight abuse of notation, refer to the energy $E_\e(\nabla u_\e)$ without making an explicit choice of $u_\e$. \par
Next, let $g$ be the local solution of the following initial value problem:
\begin{equation}\label{ivp}
\begin{cases} g'(t) = \displaystyle\frac{|g p_2 + \mm_2 -(gp_1 + \mm_1 )^2/2|}{p_1 n_1}=: W(g), \\
g'(0)=1/2.\end{cases}
\end{equation}
The vectors $p$ and $n$ are given by $p= \mp - \mm$ and $n = p /|p|$. The first components $n_1$, $p_1$ cannot be zero since $\mp$ and $\mm$ lie on the parabola $m_2 = (m_1)^2/2$. Here $W$ is chosen so that when we plug $g( \cdot / \e)$ into the ansatz \eqref{grads} and calculate the resulting energy density, equality is achieved in \eqref{divcalc}. This is the same idea as the transition layer ansatz for the Modica-Mortola problem \cite{ModMor77} or the Aviles-Giga problem \cite{AviGig87}. We collect some properties of the solution $g$, all of which follow from the facts that $W \geq 0$ and vanishes linearly at $0$ and $1$ (see for example \cite[Equation (1.21)]{Ste88}).
\begin{enumerate}[(i)]
\item The solution $g$ is increasing and exists for all time, and
\item there exist $c_1$, $c_2$ depending only on $W$ and thus on $\mp$, $\mp$, such that
\begin{equation}\label{expdecay}
|1- g(t)|\leq c_1e^{-c_2t}\textup{ as }t \to \infty \textup{ and }\quad|g(t)| \leq c_1e^{c_2t}\textup{ as }t\to -\infty
.\end{equation}
\end{enumerate}
The constants $c_i$ will implicitly change from line to line but will always depend only on $\mp$ and $\mm$ and not on $\e$.\par
We would like to define $g_\e = g(t/\e)$ in \eqref{grads}; however, $g(t /\e)$ does not satisfy the boundary conditions at $\pm 1/2$. To account for this, we will linearly interpolate if $|(x,z) \cdot \nu)| > 1/4$ and use the rescaled $g$ elsewhere. We set
\begin{equation}\notag
\nabla u_\e (x,z) =\begin{cases} \left[ g\left( \displaystyle\frac{1}{4\e}\right)+4\left(1- g\left( \displaystyle\frac{1}{4\e}\right)\right)\left(\nu \cdot (x,z) - \displaystyle\frac{1}{4}\right)\right]p + \mm  & \mbox{if } (x,z)\cdot \nu \geq \displaystyle\frac{1}{4}, \\
g\left(\displaystyle\frac{{\color{black}\nu} \cdot (x,z)}{\e} \right)p+\mm & \mbox{if } |(x,z)\cdot \nu |\leq \displaystyle\frac{1}{4},\\
\left[ g\left( -\displaystyle\frac{1}{4\e}\right)+4g\left( -\displaystyle\frac{1}{4\e}\right)\left(\nu \cdot (x,z) + \displaystyle\frac{1}{4}\right)\right]p + \mm & \mbox{if } (x,z)\cdot \nu \leq -\displaystyle\frac{1}{4}.
\end{cases}
\end{equation}
It is straightforward to check that due to \eqref{expdecay}, 
\begin{equation}\label{uest}
{\color{black}(\dxx u_\e)^2}, \left( \dz u_\e - \frac{1}{2}(\dx u_\e)^2\right)^2 \leq c_1 e^{-c_2/\e} \textup{ when } |(x,z)\cdot \nu| \geq 1/4.
\end{equation}\par
By \eqref{relation}, to finish the proof, it suffices to show that $\nabla u_\e$ satisfy
\begin{equation}\label{finalest}
E_\e ( \nabla u_\e) \leq |(\Sigma(\mp) - \Sigma(\mm))\cdot \nu | + c_1 e^{-c_2 /\e}
.\end{equation}
Let us split up the energies as
\begin{align}\notag
E_\e (\nabla u_\e) &= \frac{1}{2}\int_R \left\{ \displaystyle\frac{\left(\dz u_\e - \frac{1}{2}(\dx u_\e)^2\right)^2}{\e} + \e (\dxx u_\e)^2\right\}\,dx\,dz \\ \notag
&= \frac{1}{2}\int_{\{| (x,z) \cdot \nu |\leq 1/4 \}} \left\{ \displaystyle\frac{\left(\dz u_\e - \frac{1}{2}(\dx u_\e)^2\right)^2}{\e} + \e (\dxx u_\e)^2\right\}\,dx\,dz \\ \notag &\quad +\frac{1}{2}\int_{\{| (x,z) |\cdot \nu > 1/4 \}} \left\{ \displaystyle\frac{\left(\dz u_\e - \frac{1}{2}(\dx u_\e)^2\right)^2}{\e} + \e (\dxx u_\e)^2\right\}\,dx\,dz \\ \notag
&:= I_\e^1 + I_\e^2.
\end{align}
First, from \eqref{uest}, we have
\begin{equation}
I_\e^2 \leq c_1 e^{-c_2/\e}.
\end{equation}
For $I_\e^1$, we write
\begin{align}\notag
I_\e^1 &=\frac{1}{2}\int_{\{| (x,z) \cdot \nu |\leq 1/4 \}} \left\{ \displaystyle\frac{\left(\dz u_\e - \frac{1}{2}(\dx u_\e)^2\right)^2}{\e} + \e (\dxx u_\e)^2\right\}\,dx\,dz \\ \notag
&= \frac{1}{2}\int_{-1/4}^{1/4}\left\{  \frac{\left(g(t/\e)p_2 + \mm_2  - (g(t/\e)p_1 + \mm_1)^2/2\right)^2}{\e}+\frac{g'(t/\e)^2p_1^2 \nu_1^2}{\e} \right\} \,dt .
\end{align}
Now $\nu \parallelsum n$ and both are unit vectors, so we can substitute $\nu_1^2 = n_1^2$. Then by \eqref{ivp}, we have
\begin{align} \notag
I_\e^1 &= \frac{1}{2}\left| \int_{-1/4}^{1/4} \frac{2}{\e} g'(t/\e)p_1 n_1\left( g(t/\e)p_2 + \mm_2  - (g(t/\e)p_1 + \mm_1)^2/2\right)  \, dt\right| \\ \notag
&=\left| \int_{g(-1/(4\e))}^{g(1/(4\e))}  p_1 n_1 \left( sp_2 + \mm_2  - (sp_1 + \mm_1)^2/2\right)  \, ds\right| \\ \notag
&= \left| \int_{g(-1/(4\e))}^{g(1/(4\e))}  p_1 n_1 \left( sp_2 + \mm_2  - s^2p_1^2/2 - sp_1\mm_1 - (\mm_1)^2/2 \right)  \, ds\right| \\ \notag
&= \left| \int_{g(-1/(4\e))}^{g(1/(4\e))}  p_1 n_1 \left( sp_2   - s^2p_1^2/2 - sp_1\mm_1 \right)  \, ds\right| .
\end{align}
In the last line we use $\mm_2 = (\mm_1)^2/2$ to cancel two of the terms in parentheses. Since $g(\pm 1 / (4\e))$ approaches 1 and 0 exponentially as $\e \to 0$, we can finish the calculation:
\begin{align} \notag
I_\e^1 &\leq   \left| \int_{0}^{1}  p_1 n_1 \left( sp_2   - s^2p_1^2/2 - sp_1\mm_1 \right)  \, ds\right| + c_1 e^{-c_2 / \e} \\ \notag
&= \left|\frac{n_1}{2}(p_1p_2 - p_1^3/3 - p_1^2\mm_1 )\right| + c_1 e^{-c_2 / \e} \\ \notag
&= \left|(\Sigma(\mp) - \Sigma(\mm))\cdot \nu \right|+ c_1 e^{-c_2 / \e} \quad \textup{by \Cref{jumpexp}}.
\end{align}
Thus
\begin{equation}\notag
E_\e(\nabla u_\e) = I_\e^1 + I_\e^2 \leq \left|(\Sigma(\mp) - \Sigma(\mm))\cdot \nu \right|+ c_1 e^{-c_2 / \e},
\end{equation}
as desired.
\end{proof}\par
Finally, we can prove the upper bound \Cref{2dupbd}. We quote a theorem from \cite{Pol08} which is valid in any dimension and for a wide range of energy densities. Let us write the version that applies to our problem.
\begin{theorem}[Theorem 1.2 from \cite{Pol08}]\label{polthm}
Let $\Omega$ be a bounded $C^2$-domain and let 
$$
F(a,b) : \mathbb{R}^{2\times 2} \times \mathbb{R}^2  \to \mathbb{R}
$$
be a $C^1$ function satisfying $F \geq 0$. Let $u \in W^{1,{\infty}}(\Omega)$ be such that $\nabla u \in BV(\Omega;\mathbb{R}^2)$ and $F(0,\nabla u(x))=0$ a.e. in $\Omega$. Then there exists a family of functions $\{ u_\e \} \subset C^2(\mathbb{R}^2)$ satisfying
\begin{equation}\notag
u_\e \to u \textit{ in }W^{1,{p}}(\Omega) \textit{ for }1\leq p < \infty
\end{equation}
and
\begin{align}\notag
&\lim_{\e \to 0} \frac{1}{\e} \int_\Omega F(\e \nabla^2 u_\e, \nabla u_\e)\,dx \,dz \\ \notag
&=\int_{J_{\nabla u}} \inf_{r \in \mathcal{R}_{\chi(x,z),0}} \left\{ \int_{-\infty}^\infty F\left(-r'(t) \nu(x,z) \otimes \nu(x,z), r(t)\nu(x,z)+\nabla u^-(x,z)\right) \,dt\right\}d\mathcal{H}^1.
\end{align}
Here $\chi(x,z)$ is given by
\begin{equation}\notag
\chi(x,z) \nu(x,z) = \nabla u^+(x,z)-\nabla u^-(x,z),
\end{equation}
and
$$
\mathcal{R}_{\chi(x,z),0}:=\{r(t) \in C^1(\mathbb{R}): \exists L>0 \textit{ s.t. } r(t)= \chi(x,z)  \textit{ for }t\leq -L, r(t) = 0\textit{ for }t\geq L\}
.$$
\end{theorem}
\begin{proof}[Proof of \Cref{2dupbd}]
In the notation of the statement of \Cref{polthm}, we take
\begin{equation}\notag
F(a,b) = (b_2 - b_1^2/2){\color{black}^2}/2 + a_{11}^2/2,
\end{equation}
so that
\begin{equation}\notag
\frac{1}{\e} \int_\Omega F(\e \nabla^2 u_\e, \nabla u_\e)\,dx \,dz = E_\e (u_\e).
\end{equation}
By rescaling and applying \Cref{boxrate}, we find that the infimum in \Cref{polthm} is in fact $|(\Sigma(\nabla u^+) - \Sigma(\nabla u^-))\cdot \nu|$. The proof of \Cref{2dupbd} is complete.
\end{proof}
\begin{remark}
\Cref{polthm} does not deal with specifying boundary conditions for the recovery sequence, and so we have not included this in our analysis either. In the Aviles-Giga problem, this was handled in \cite[Theorem 1.1]{Pol07} and \cite[Section 6]{ConDeL07}, and so similar techniques could apply here as well. 
\end{remark}
{\color{black}\begin{remark}
Ignat and Monteil \cite[Proposition 4.19]{IgnMon20} introduced a systematic approach for determining the 1d symmetry for minimizers of variational integrals of the form
\begin{equation*}
E(m)=\int\frac{1}{2}|\nabla m|^2+W(m)  \text{ with } \nabla \cdot m=0.
\end{equation*}
Denoting by $\Pi_0$ the projection onto traceless matrices, the map $\Phi$ defined in \Cref{entremark} satisfies the ``strong punctual condition" 
\begin{equation*}
| \Pi_0 \nabla \Phi(m)|^2 \leq W(m)
\end{equation*}
from \cite{IgnMon20}, which underlies the fact that the divergence of the entropy bounds the energy from below. In the case where the limiting jump for $\nabla u$ has normal vector $\hat x$, the results of \cite{IgnMon20} then yield the optimality and uniqueness of the 1d profile. 
\end{remark}
}
\section{Appendix}  We prove \Cref{jumpexp}.
\begin{proof}[Proof of \Cref{jumpexp}]
First, we prove
\begin{equation}\label{firstexp}
(\Sigma(\mp) - \Sigma(\mm)) \cdot n = \frac{n_1}{2} \left(p_1p_2-m^-_1p_1^2 -\frac{1}{3}p_1^3 \right).
\end{equation}
Let us record the identities
\begin{align}\label{idents}
n_1 p_2 = n_2 p_1 ,\quad \textup{and}\quad\mm_2 = (\mm_1)^2/2.
\end{align}
We calculate
\begingroup
\allowdisplaybreaks
\begin{align}\notag
(\Sigma(m^+) - \Sigma(m^-))\cdot n &= \left( \mp_1 \mp_2 - \frac{1}{6}(\mp_1)^3 \right)n_1 - \frac{1}{2}(\mp_1)^2n_2\\ \notag  &\quad- \left( \mm_1 \mm_2 - \frac{1}{6}(\mm_1)^3 \right)n_1 + \frac{1}{2}(\mm_1)^2n_2 \\ \notag
&= n_1 \left[\mp_1\mp_2 - \mm_1\mm_2 - \frac{1}{6}(\mp_1)^3 +\frac{1}{6}(\mm_1)^3\right] \\ \notag &\quad+ \frac{1}{2}n_2 \left[(\mm_1)^2 - (\mp_1)^2\right] \\ \notag
&= n_1 \left[(\mm_1+p_1)(\mm_2+p_2) - \mm_1\mm_2 - \frac{1}{6}\left(\mm_1+p_1 \right)^3 + \frac{1}{6}(\mm_1)^3\right] \\ \notag
&\quad+ \frac{1}{2}n_2 \left[(\mm_1)^2 - (\mm_1+p_1)^2 \right] \\ \notag
&=  n_1 \left[p_1p_2 + p_1\mm_2 + p_2 \mm_1- \frac{1}{2}(\mm_1)^2p_1- \frac{1}{2}\mm_1p_1^2 - \frac{1}{6}p_1^3\right] \\ \notag
&\quad+ \frac{1}{2}n_2 \left[-2p_1\mm_1 - p_1^2 \right] .
\end{align}
Notice that the second and fourth terms in the first bracket add to 0 by the second identity in \eqref{idents}. Continuing on and then using the first identity in \eqref{idents} to cancel $n_1\mm_1p_2 - n_2p_1\mm_1$, we have
\begin{align}\notag
(\Sigma(m^+) - \Sigma(m^-))\cdot n&= n_1 \left[p_1p_2 + p_2 \mm_1- \frac{1}{2}\mm_1p_1^2 - \frac{1}{6}p_1^3\right] + \frac{1}{2}n_2 \left[-2p_1\mm_1 - p_1^2 \right]  \\ \notag
&= n_1\left[ p_1p_2  -\frac{1}{2}\mm_1  p_1^2 - \frac{1}{6}p_1^3\right]-  \frac{1}{2}n_2p_1^2.
\end{align}
Finally, we finish the proof of \eqref{firstexp} by again using \eqref{idents} to rewrite $n_1 p_1 p_2 - n_2 p_1^2/2$ as $n_1 p_1 p_2 - n_1 p_1p_2/2=n_1p_1p_2/2$, which gives
\begin{align}\notag
(\Sigma(m^+) - \Sigma(m^-))\cdot n = n_1\left[\frac{1}{2} p_1p_2  - \frac{1}{2}\mm_1 p_1^2- \frac{1}{6}p_1^3\right] .
\end{align}
\endgroup
\par
Moving on to the second expression for $(\Sigma(m^+) - \Sigma(m^-))\cdot n$, we show
\begin{equation}\label{secondexp}
(\Sigma(\mp) - \Sigma(\mm)) \cdot n = \displaystyle\frac{|m_1^+ - m_1^-|^3}{12\sqrt{1 +\frac{1}{4}(m_1^+ + m_1^-)^2}}.
\end{equation}
The calculation is straightforward. We write $n$ as
\begin{align*}\notag
n = \frac{\mp - \mm}{|\mp-\mm|}=\frac{\left(\mpo,\frac{1}{2}(\mpo)^2\right)-\left(\mmo,\frac{1}{2}(\mmo)^2\right)}{\sqrt{(\mpo-\mmo)^2+\frac{1}{4}\left((\mpo)^2-(\mmo)^2\right)^2}}
\end{align*}
and notice that if $m_2 = \frac{1}{2}m_1^2$, then
\begin{equation*}
\Sigma(m) = \left(\frac{1}{3}m_1^3, -\frac{1}{2} m_1^2 \right).
\end{equation*}
Thus
\begingroup
\allowdisplaybreaks
\begin{align*}
(\Sigma&(\mp) - \Sigma(\mm)) \cdot n \\ 
&= \left( \frac{1}{3}(\mpo)^3 - \frac{1}{3}(\mmo)^3, \frac{1}{2}(\mmo)^2 - \frac{1}{2}(\mpo)^2\right)\cdot \frac{\left(\mpo,\frac{1}{2}(\mpo)^2\right)-\left(\mmo,\frac{1}{2}(\mmo)^2\right)}{\sqrt{(\mpo-\mmo)^2+\frac{1}{4}\left((\mpo)^2-(\mmo)^2\right)^2}} \\
&= \frac{(\mpo-\mmo)}{2}\left(\frac{2}{3}\left((\mpo)^2+\mpo\mmo + (\mmo)^2\right), -(\mpo+\mmo) \right) \\
& \quad\cdot \frac{(\mpo-\mmo)\left(1,\frac{1}{2}(\mpo+\mmo)\right)}{|\mpo-\mmo|\sqrt{1+\frac{1}{4}(\mpo+\mmo)^2}}\\
&= \frac{(\mpo-\mmo)^2}{2|\mpo-\mmo|}\frac{\left(\frac{2}{3}(\mpo)^2 + \frac{2}{3}(\mpo\mmo) + \frac{2}{3}(\mmo)^2 - \frac{1}{2}(\mpo+\mmo)^2 \right)}{\sqrt{1+\frac{1}{4}(\mpo+\mmo)^2}}\\
&= \frac{|\mpo-\mmo|}{2}\frac{\left(\frac{1}{6}(\mpo)^2 + \frac{1}{6}(\mmo)^2 - \frac{1}{3}\mpo\mmo \right)}{\sqrt{1+\frac{1}{4}(\mpo+\mmo)^2}}\\
&= \frac{|\mpo-\mmo|^3}{12\sqrt{1+\frac{1}{4}(\mpo+\mmo)^2}}.
\end{align*}
\endgroup
\end{proof}
\color{black}

\textbf{ACKNOWLEDGMENTS}
We thank Robert V. Kohn for bringing this problem into our interest. {\color{black} We also thank both referees and the editor for helpful comments. M. N.'s research is supported by NSF grant RTG-DMS 1840314.} X.Y's
research is supported by a Research Excellence Grant from University of
Connecticut.

\FloatBarrier 
\bibliographystyle{siam}
\bibliography{references}
\end{document}